\documentclass[10pt,reqno]{amsart}
\usepackage{amsmath,amsthm,mathrsfs,amsfonts,amssymb,amscd,euscript,color}

\newtheorem{theorem}{Theorem}[section]
\newtheorem*{theorem*}{Theorem}
\newtheorem{lemma}[theorem]{Lemma}

\newtheorem{corollary}[theorem]{Corollary}
\newtheorem{claim}[theorem]{Claim}
\newtheorem{proposition}[theorem]{Proposition}

\newtheorem{maintheorem}{Theorem}

\theoremstyle{definition}
\newtheorem{definition}[theorem]{Definition}
\newtheorem{example}{Example}
\newtheorem*{example*}{Example}
\newtheorem{remark}[theorem]{Remark}
\newtheorem*{remark*}{Remark}

\newcommand{\eqdef}{\stackrel{\scriptscriptstyle\rm def}{=}}

\def\cA{\mathscr A}
\def\cB{\mathscr{B}}

\def\s{{\rm s}}
\def\u{{\rm u}}

\def\H{{\rm H}}
\DeclareMathOperator{\DD}{DD}

\DeclareMathOperator{\hD}{hD}

\DeclareMathOperator{\diam}{diam}

\DeclareMathOperator{\supp}{supp}

\def\bC{\mathbb{C}}
\def\bH{\mathbb{H}}
\def\bN{\mathbb{N}}
\def\bS{\mathbb{S}}
\def\bT{\mathbb{T}}
\def\bZ{\mathbb{Z}}
\def\bR{\mathbb{R}}

\def\cC{\mathscr{C}}

\def\cF{\mathcal{F}}
\def\cR{\mathcal{R}}
\def\cH{\mathcal{H}}
\def\cO{\EuScript{O}}

\def\cR{\mathscr{R}}

\def\cH{\mathscr{H}}
\def\cU{\EuScript{U}}

\def\cW{\mathscr{W}}
\def\cM{\EuScript{M}}

\DeclareMathSymbol{\varnothing}{\mathord}{AMSb}{"3F}
\author[S. Campos]{Sara Campos}\address{Department of Mathematics, Federal University of Juiz de Fora, Campus Universit\'ario - Bairro Martelos, Juiz de Fora  36036-900, MG, Brazil}\email{sara.campos@edu.ufjf.br}
\author[K. Gelfert]{Katrin Gelfert}\address{Institute of Mathematics, Federal University of Rio de Janeiro, Av. Athos da Silveira Ramos 149, Cidade Universit\'aria - Ilha do Fund\~ao, Rio de Janeiro 21945-909, RJ, Brazil}\email{gelfert@im.ufrj.br}

\begin{document}

\title[Exceptional sets]{Exceptional sets\\for nonuniformly hyperbolic diffeomorphisms}

\begin{abstract}
For a surface diffeomorphism, a compact invariant locally maximal set $W$ and some subset $A\subset W$ we study the $A$-exceptional set, that is, the set of points whose orbits do not accumulate at $A$.  We show that if the Hausdorff dimension of $A$ is smaller than the Hausdorff dimension $d$ of some ergodic hyperbolic measure, then the topological entropy of the exceptional set is at least the entropy of this measure and its Hausdorff dimension is at least $d$. Particular consequences occur when there is some \emph{a priori} defined hyperbolic structure on $W$ and, for example, if there exists an SRB measure.
\end{abstract}

\begin{thanks}{This research has been supported by CNPq research grant 302880/2015-1 (Brazil). KG thanks M.~Rams and Ch.~Wolf for comments on measures of maximal dimension and T.~J\"ager on Proposition~\ref{pro:Manning}. SC thanks A.~Arbieto for helpful discussions.}\end{thanks}
\keywords{topological entropy, Hausdorff dimension, exceptional sets}
\subjclass[2000]{%
37B40, 
37C45, 
37D25, 
37F35, 
}
\maketitle
\tableofcontents
\section{Introduction}

The study of orbits of hyperbolic torus automorphisms is a very classical field. Any linear automorphism given by a $n\times n$-integer matrix of determinant $\pm1$ and without eigenvalues of absolute value $1$ induces a hyperbolic  automorphism of the torus $\bT^n=\bR^n/\bZ^n$ and provides the simplest example of an Anosov diffeomorphism. One of their most important features is their ergodicity (with respect to the Haar measure) which implies in particular that almost all points have a dense orbit. Nevertheless, the complementary often called \emph{exceptional set}, that is, the set of points with non-dense orbit, can in general be quite large. 

In this paper we are interested in the ``size" of exceptional sets in terms of their topological entropy and Hausdorff dimension. We will study limit exceptional sets for surface diffeomorphisms.

Let us  introduce some notation. Given a  metric space $(X,d)$ and a continuous transformation $f\colon X\to X$, we denote by $\cO^+_{f|X}(x)\eqdef\{f^k(x)\colon k\in\bN\cup\{ 0 \} \}$ the (\emph{forward}) \emph{semi-orbit} of $x\in X$  by $f$. 
We say that a set $Y\subset X$ is  \emph{forward $f$-invariant} if $f(Y)\subset Y$ and \emph{$f$-invariant} if $f(Y)=Y$. Denote by $\omega_f(x)$ the \emph{(forward) $\omega$-limit set} of a point $x\in X$, that is, the set of limit points of $\cO^+_{f|X}(x)$. Denote by $\overline Y$ the \emph{closure} of a set $Y\subset X$.

\begin{definition}[Exceptional set]
Given a set $A\subset X$, the \emph{(forward) $A$-exceptional set}  (with respect to $f$) is defined by
\[
	E^+_{f|X}(A)
	\eqdef \{x\in X\colon \overline{\cO^+_{f|X}(x)}\cap A=\emptyset\}
\]
and the \emph{(forward) limit $A$-exceptional set} (with respect to $f$)  is defined by  
\[
	I^+_{f|X}(A)
	\eqdef \{x\in X\colon \omega_f(x)\cap A=\emptyset\}.
\]
\end{definition}

\begin{remark}\label{remarkIE}
Note that $E^+_{f|X}(A)\subset I^+_{f|X}(A)$ and that $I^+_{f|X}(A)$ is $f$-invariant while $E^+_{f|X}(A)$ is forward $f$-invariant. 
Observe also that 
\[
	I^+_{f|W}(A) = E^+_{f|W}(A) \cup\bigcup_{n\ge0}f^{-n}( \tilde{A}),
	\quad\text{ where }\quad
	\tilde{A} \eqdef \{a\in A\colon \omega_f(a)\cap A = \emptyset\}.
\]	
 Note that this union is disjoint. 
\end{remark}

\subsection{Main results}

 Unless otherwise stated, in this paper $f\colon M\to M$ will be always  a $C^{1+\varepsilon}$ diffeomorphism of a Riemannian surface $M$ and $W\subset M$ some compact $f$-invariant locally maximal set. This includes the possibilities of either a hyperbolic ergodic measure whose support is a locally maximal set (see Theorem~\ref{the:smallsmall}), a basic set of a surface diffeomorphism (see Theorem~\ref{the:basic}), or  an Anosov surface diffeomorphism (see Theorem~\ref{the:anosov}). 
Recall that a set $W\subset M$ is \emph{locally maximal} (or \emph{isolated}) if there exists a neighborhood $U$ of $W$ such that
\[
	W
	= \bigcap_{k\in\bZ}f^k(U).
\]

We denote by $\dim_\H( B)$ the \emph{Hausdorff dimension}  of a set $B\subset M$ (see~\cite{Fal:03}) and by $h(f|_W,B)$ the \emph{topological entropy} of $f|_W$ on $B\subset W$ (we briefly recall their definitions in Sections~\ref{sec:dis} and~\ref{sec:entropies}, respectively). 

The following is our first main result.

\begin{maintheorem}\label{teo:mainentropy}
	Let $f\colon M\to M$ be a $C^{1+\varepsilon}$ diffeomorphism of a compact Riemannian surface. Let $W\subset M$ be a compact $f$-invariant locally maximal set. 
	
	For every $A\subset W$ such that $h(f|_W,A)<h(f|_{W})$, we have
\[
	h(f|_{W},I^+_{f|W}(A))
	= h(f|_{W},E^+_{f|W}(A))
   	= h(f|_{W}).
\]	 
\end{maintheorem}

To state our second main result, denote by $\cM=\cM(f|_W)$ the space of all $f$-invariant Borel probability measures supported on $W$ and by $\cM_{\rm erg}\subset\cM$ the subset of all ergodic measures. 
Given $\mu\in\cM$, define the \emph{Hausdorff dimension} of $\mu$ by
\[
	\dim_\H\mu
	\eqdef \inf\{\dim_\H( B)\colon B\subset M\text{ and }\mu(B)=1\}.
\]
We denote by $h_\mu(f)$ the \emph{entropy} of $\mu$. Note that (since we consider a surface diffeomorphism $f$) every ergodic measure $\mu$ with positive entropy is hyperbolic (we recall \emph{hyperbolicity} in Section~\ref{sec:horse}). 

\begin{maintheorem}\label{the:smallsmall}
	Let $f\colon M\to M$ be a $C^{1+\varepsilon}$ diffeomorphism of a compact Riemannian surface and let $\mu\in\cM$ be a hyperbolic $f$-invariant ergodic measure whose support $W\eqdef\supp\mu$ is locally maximal. 

	For every $A\subset W$ such that $\dim_\H( A)<\dim_\H\mu$, we have
\[
	h(f|_W,I^+_{f|W}(A))
	\ge h_\mu(f)
	\quad\text{ and }\quad
	\dim_\H( E^+_{f|W}(A))
	\ge \dim_\H\mu.
\]	
\end{maintheorem}

\begin{remark}
	Note that it may happen that the Hausdorff dimension of a measure is smaller than the one of its support. Indeed, Example~\ref{ex:2} in Section~\ref{sec:hyp} provides such a case and shows that the second inequality in Theorem~\ref{the:smallsmall} can be strict (see also Theorem~\ref{the:anosov} below).
\end{remark}

\begin{remark}
Note that the hypotheses that $\mu$ is hyperbolic and that $\dim_\H\mu>0$ (which by Young's formula~\eqref{eq:Young} is equivalent to $h_\mu(f)>0$) in Theorem~\ref{the:smallsmall} automatically exclude that $\mu$ is supported on a hyperbolic periodic orbit. Note that the hypothesis $\mu$ being hyperbolic is necessary for the  conclusion of Theorem~\ref{the:smallsmall}. Indeed, if $f\colon\bT^2\to\bT^2$ is a minimal diffeomorphism such that the Haar measure $\mu$ is $f$-ergodic, then for any $A=\{x\}$, $x\in\bT^2$, the (limit) $A$-exceptional set is empty and that $\dim_\H\mu=2$.
\end{remark}

To state our third main result, we define the {\it dynamical dimension} of $f|_{ W}$  by   
\begin{equation}\label{def:DD}
	\DD(f|_W) \eqdef  \sup_\mu\dim_\H \mu,
\end{equation}
where the supremum is taken over all ergodic measures $\mu\in\cM_{\rm erg}$  with positive entropy. In Section~\ref{sec:hyp} we discuss some of its properties and  provide examples where $\DD(f|_W)<\dim_\H( W)$. Recall that for any such measure  by Young's formula~\cite{You:82}  
\begin{equation}\label{eq:Young}
	\dim_\H \mu
	= h_\mu(f)\Big(\frac{1}{\chi^\u(\mu)}-\frac{1}{\chi^\s(\mu)}\Big)	,
\end{equation}
where  $\chi^\s(\mu)<0<\chi^\u(\mu)$ denote the Lyapunov exponents of $\mu$ (see~\cite{KatHas:95} for definition and details and Section~\ref{sec:horse}). In particular,  if the topological entropy of $f|_{W}$ is positive then $\DD(f|_{W})>0$. 

\begin{maintheorem}\label{main}
	Let $f\colon M\to M$ be a $C^{1+\varepsilon}$ diffeomorphism of a compact Riemannian surface. Let $W\subset M$ be a compact $f$-invariant locally maximal set. 
	
	For every $A\subset W$ such that $\dim_\H( A)<\DD(f|_{ W})$, we have 
\[
	\dim_\H( E^+_{f|W}(A) )
	\ge \DD(f|_W).
\]
\end{maintheorem}

\subsection{Previous results on exceptional sets and related topics}

The interest in exceptional sets has many origins and it started with the work of Jarnik-Besicovitch~\cite{Jar:29}: Recall that a real number $x$ is \emph{badly approximable} if there is a positive constant $c=c(x)$ such that for any reduced rational $p/q$ we have $\lvert p/q-x\rvert > c/q^2$. Looking from an algebraic point of view, Jarnik's theorem states that the Hausdorff dimension of  the set of badly approximable numbers in the unit interval is $1$. 

In view of our main results it is worth mentioning the following point of view of Jarnik's theorem and its generalizations. Namely it can be equivalently read in terms of bounded geodesic curves emanating from a fixed point of the surface $M=\bH^2/SL(2,\bZ)$, where $SL(2,\bZ)$ denotes the group of $2\times 2$ matrices with integer entries and where one projects the Poincar\'e metric of the hyperbolic plane $\bH^2$ to its quotient. A geodesic on $\bH^2$ is bounded if and only if its end points in $\bS^1=\partial\bH^2$ are badly approximable. Thus, one can conclude that the set of directions such that the corresponding geodesic is bounded has Hausdorff dimension $1$.  
This result was generalized to complete noncompact manifolds of negative curvature and finite volume (see, for example \cite{Dan:89}) and to  many more general contexts yielding the same type of result that the set of directions with bounded geodesics has full Hausdorff dimension, that is, has Hausdorff dimension equal to the one of the subset of recurrent directions (those whose forward and backward geodesic rays intersect infinitely often some compact region, respectively).

From a slightly different point of view, Hirsch~\cite{Hir:70}  suggested to exhibit general properties which are common for all compact invariant sets of a hyperbolic torus automorphism. According to~\cite[p.~134]{Hir:70}, Smale showed that for an automorphism of $\bT^2$ there is no nontrivial compact invariant one-dimensional set.
 In addressing these points, Franks~\cite{Fra:77} showed that, given any $C^2$ nonconstant  curve $\gamma\colon(a,b)\to\bT^n$ and a hyperbolic torus automorphism $f\colon\bT^n\to\bT^n$, the set $\gamma((a,b))$ contains a point whose orbit by $f$ is dense in $\bT^n$. Ma\~n\'e~\cite{Man:79} extended this result to rectifiable%
\footnote{Recall that a curve $\gamma\colon(a,b)\to\bT^n$ is \emph{rectifiable} if it is continuous and if there exists a constant $C>0$ satisfying $\sum_nd(\gamma(t_{n+1},\gamma(t_n))\le C$ for all partitions $a=t_0\le t_1\le\ldots\le t_{n+1}=b$.  It is interesting to observe that Hancock~\cite{Han:78} provides examples that show that Ma\~n\'e's result does not extend to continuous curves. Note that a rectifiable path $\gamma((a,b))$ has Hausdorff dimension $1$ and that a merely \emph{continuous} path can have Hausdorff dimension $>1$ (see, for example,~\cite[Chapter 11]{Fal:03}).} 
curves.
Thus, on one hand no nontrivial invariant set can be a manifold. On the other hand a $f$-exceptional set cannot contain any rectifiable path and the question about the fractal nature of such sets arises.  

Looking again at Jarnik's theorem, it is not difficult to see that a number $x\in[0,1)$ is badly approximable if and only if  the closure of the semi-orbit of $x$ under the Gauss map $f\colon[0,1)\to[0,1)$ (and hence the set of its limit points) does not contain the point $0$), that is, a number $x\in[0,1)$ is badly approximable if and only if $x$ is in the exceptional (and hence in the limit exceptional) set of $\{0\}$  (with respect to the Gauss map $f$).
These results motivate the general question about the Hausdorff dimension of the $A$-exceptional set for some ``sufficiently small'' set of points $A$.
 Abercrombie and Nair proved in~\cite{AN:interval} a version of Jarnik's theorem  for interval Markov maps. Dani investigated special countable subsets $A\subset\bT^n$ and showed that the $A$-exceptional set under a hyperbolic torus automorphism has full Hausdorff dimension $n$ (see~\cite[Corollary 2.7]{Dan:88}).%
\footnote{In fact, Dani in~\cite{Dan:88} and also in the before mentioned article~\cite{Dan:89} consideres a more general setting and obtained a stronger conclusion that such sets are \emph{winning} in the sense of Schmidt games.
See for example the original work by Schmidt who showed that any winning set has full Hausdorff dimension (see~\cite[Section 11]{Sch:66}). Though Schmidt games so far were mainly applied to questions of algebraic nature
(see, for example, the introduction of~\cite{Wu:16} for numerous references), more recently they were also used to investigate the fractal structure of exceptional sets (see~\cite{Tse:09,Wu:16}).}
 
In the context of expanding dynamical systems, exceptional sets were also studied for example by Urba\'nski~\cite{Urb:91} considering a $C^2$ expanding map of a Riemannian manifold $X$ and showing that for any $x\in X$ the forward $\{x\}$-exceptional set has full Hausdorff dimension equal to the dimension of $X$. Abercrombie and Nair~\cite{AbeNai:97} considered expanding rational maps of the Riemann sphere on its Julia set and the forward $A$-exceptional set of a \emph{finite} set of points $A$ and also established that this set has full Hausdorff dimension, that is, dimension equal to the Hausdorff dimension of the Julia set. Their approach is based on a construction of a  Borel measure supported on the set of points whose forward orbit misses certain neighborhoods of $A$ and the use of the mass distribution principle to determine dimension. 

Ideas similar to~\cite{Urb:91,AbeNai:97} were also used by Dolgopyat~\cite{Dol:97} where exceptional sets are studied in several contexts: a one-sided shift space (see Theorem~\ref{pro1:Dol}), piecewise uniformly expanding maps of the interval, Anosov surface diffeomorphisms (see Example~\ref{ex:2}),  conformal Anosov flows, and geodesic flows of Riemannian surfaces of negative curvature. In general terms he showed that for any set $A$ which is ``sufficiently small'' in the sense that it has  small topological entropy or small Hausdorff dimension compared to the one of the dynamical system, the $A$-exceptional set is ``large'' in the sense that it has  full entropy or full Hausdorff dimension, respectively. His proofs are also based on the construction of a certain Borel measure and applying the mass distribution principle; in all his classes of systems the possibility of symbolic representation of the dynamics facilitates the construction of such measures. 

To follow the approaches in~\cite{AbeNai:97,Dol:97} in a nonhyperbolic context is more difficult. In~\cite{CamGel:16}, the model case of rational maps of the Riemann sphere on its Julia sets was studied, including  the cases of maps with critical points or parabolic points and corresponding results were obtain. Here the condition ``sufficiently small" means that the Hausdorff dimension of $A$ has to be smaller than the \emph{dynamical dimension} of the system, that is, the maximal Hausdorff dimension of ergodic measures with positive entropy (compare~\eqref{def:DD}). The approach in~\cite{CamGel:16} is, instead of studying the dynamical systems on the whole, to consider appropriate sub-dynamical systems which are uniformly expanding and hence allow to apply the abstract result on shift spaces in~\cite{Dol:97} and gradually approximate ``from inside" the full dynamics (in the setting of this paper we  will proceed analogously, see Section~\ref{sec:horse}). 
In this paper we adapt approach in~\cite{CamGel:16} to our setting. 

\subsection{Improved results in specific cases}

To improve the lower bound in Theorem~\ref{main} in some specific cases, we require \emph{a priori} information about a hyperbolic structure on the whole set $W$. We are going to provide some examples.

We first recall some concepts (see~\cite{KatHas:95}). 
Given a diffeomorphism $g\colon M\to M$ and a compact invariant set $\Gamma\subset M$, we say that $\Gamma$ is \emph{hyperbolic} if (up to a change of metric) there exist a  $dg$-invariant splitting $E^\s\oplus E^\u=T_\Gamma M$ and numbers $0<\mu<1<\kappa$ such that for every $x\in\Gamma$ we have
\[
	\lVert dg_{/E^\s_x}\rVert\le \mu<1<
	\kappa\le \lVert dg_{/E^\u_x}\rVert.
\]
We say that $g\colon M\to M$ is \emph{Anosov} if $M$ is hyperbolic. 
Recall that a set $\Gamma\subset M$ is \emph{basic} (with respect to $g$) if it is compact, invariant, locally maximal, and hyperbolic and  if $g|_\Gamma$ is topologically mixing  (see~\cite[Chapter 6.4]{KatHas:95} for more details).

In view of the definition of the dynamical dimension in~\eqref{def:DD} and Young's formula~\eqref{eq:Young}, the following result improves the lower bound provided in Theorem~\ref{main} in case that $W$ is a basic set.

\begin{maintheorem}\label{the:basic}
	Let $f\colon M\to M$ be a $C^{1+\varepsilon}$ diffeomorphism of a compact Riemannian surface. Let $W\subset M$ be a basic set (with respect to $f$) and let $\mu$ be a $f$-invariant ergodic measure supported on $W$.
	
	For every $A\subset W$ such that $\dim_\H( A)<\dim_\H\mu$, we have
\[
	\dim_\H( I^+_{f|W}(A))
	\ge d^\s(W) + \frac{h_\mu(f)}{\chi^\u(\mu)}
	\quad\text{ where }\quad
	d^\s(W)
	\eqdef \max_{\nu\in\cM_{\rm erg}(f|W)}\frac{h_\nu(f)}{\lvert\chi^\s(\nu)\rvert}
\]	
\end{maintheorem}

Recall that an ergodic $f$-invariant measure is  \emph{SRB} (with respect to $f$) if it has absolutely continuous conditional measures on unstable manifolds (see~\cite{You:02} for more details), we will denote it by $\mu_{\rm SRB}^+$. 
Recall that by Pesin's formula~\cite{Pes:77} we have 
\begin{equation}\label{pesfor}
	h_{\mu_{\rm SRB}^+}(f)
	= \chi^\u(\mu_{\rm SRB}^+).
\end{equation}
Theorem~\ref{the:basic} hence immediately implies the following.

\begin{corollary}
	Under the assumptions of Theorem~\ref{the:basic} and assume that there exists an SRB measure $\mu_{\rm SRB}^+\in\cM_{\rm erg}(W)$, for every $A\subset W$ such that $\dim_\H( A)<\dim_\H\mu_{\rm SRB}^+$ we have
\[
	\dim_\H( I^+_{f|W}(A))
	\ge d^\s(W) + 1.
\]	
\end{corollary}

In case of an Anosov map of a surface $M$ and $W=M$, we can state a result slightly stronger than Theorem~\ref{the:basic} (note that in this case we only know $d^\s(M)\le 1$ in general).

\begin{maintheorem}\label{the:anosov}
	Let $f\colon M\to M$ be an Anosov $C^{1+\varepsilon}$ of a compact Riemannian surface and let $\mu$ be a $f$-invariant ergodic measure.
	
	For every $A\subset M$ such that $\dim_\H( A)<\dim_\H\mu$ we have
\[
	\dim_\H( I^+_{f|M}(A) )\ge 1+\frac{h_\mu(f)}{\chi^\u(\mu)}.
\]		
\end{maintheorem}

The following result is then an immediate consequence of Theorem~\ref{the:anosov}. Note that it generalizes~\cite[Theorem 3]{Dol:97} stated for an Anosov diffeomorphism of $\bT^2$  (see also Example~\ref{ex:2}). 

\begin{corollary}
	Let $f\colon M\to M$ be an Anosov $C^{1+\varepsilon}$ of a compact Riemannian surface. For every $A\subset M$ such that $\dim_\H( A)<\dim_\H\mu_{\rm SRB}^+$ we have
\[
	\dim_\H( I^+_{f|W}(A))
	= 2.
\]	
\end{corollary}

\subsection{Essential ingredients for our proofs and organization}

We explain briefly some of the main observations which are fundamental for our proofs.

For that recall first that for a compact invariant hyperbolic set $\Gamma\subset M$ the \emph{stable manifold} of $x\in\Gamma$ (with respect to $f$) is defined by
\[\begin{split}
	\cW^\s(x,f)
	&\eqdef \{y\in M\colon d(f^{n}(y),f^{n}(x))\to 0\text{ if }n\to\infty\}.
\end{split}\]
Note that it is an injectively immersed $C^1$ one-dimensional manifold tangent to $E^\s$ on $\Gamma$. 
The \emph{local stable manifold} of  $x\in \Gamma$ (with respect to $f$ and a neighborhood $U$ of $\Gamma$) is the set
\[
	\cW^\s_{\rm loc}(x,f)
	\eqdef \big\{ y\in \cW^\s(x,f) \colon f^{k}(y)\in U\text{ for every }k\ge 0\big\}.
\]
Note that there exists $\delta>0$ such that for every $x\in\Gamma$ the local stable manifold of $x$ contains a $C^1$ stable disk centered at $x$ of radius $\delta$. The  \emph{unstable manifold} at $x$, $\cW^\u(x,f)$, and the \emph{local unstable manifold} at $x$, $\cW^\u_{\rm loc}(x,f)$, are defined analogously considering $f^{-1}$ instead of $f$. 

 First, we make the  crucial observation in Lemma~\ref{lem:locally} (which is an immediate consequence of the definition of a  limit exceptional set) that locally a limit exceptional set with respect to the dynamics in some hyperbolic set (where local stable manifolds are well defined) is a union of subsets of stable manifolds.

\begin{definition}[$\s$-saturated]
Given a hyperbolic set $\Gamma\subset M$, we call a set $B\subset\Gamma$ \emph{$\s$-saturated} (with respect to $f|_\Gamma$) if for every $x\in B$ we have $\Gamma\cap\cW^\s_{\rm loc}(x,f)\subset B$. 
\end{definition}

\begin{lemma}\label{lem:locally}
	For every $A\subset W$ and every hyperbolic set $\Gamma\subset W$ the set $I^+_{f|\Gamma}(A\cap\Gamma)$ is $\s$-saturated and invariant (with respect to $f|_\Gamma$).
\end{lemma}

\begin{proof}
	Note that for every $y\in\Gamma\cap\cW^\s_{\rm loc}(x,f)$ we have $d(f^n(y),f^n(x))\to0$ and hence $\omega_f(y)\subset\omega_f(x)$. Thus, if $x\in I^+_{f|\Gamma}(A\cap \Gamma)$ then $y\in I^+_{f|\Gamma}(A\cap\Gamma)$. The invariance follows immediately from continuity of $f$.
\end{proof}

The second key  observation is that  entropy of some basic set is the same in any intersection with a local unstable manifold. It can be seen as version%
\footnote{Manning~\cite{Man:81} considers the Hausdorff dimension of the set $B=G_\mu$ of forward $\mu$-generic points for some ergodic hyperbolic measure $\mu$ and its unstable sections.}, and its proof is very similar to the one,  of~\cite[Theorem]{Man:81}. For completeness we will include it (see Section~\ref{sec:entropies}).

\begin{proposition}\label{pro:Manning}	
	Let $f\colon M\to M$ be a $C^1$ surface diffeomorphism with a basic set $\Gamma\subset M$.
	Let $B\subset\Gamma$ be a $\s$-saturated  invariant set.
	
	Then for every $x\in\Gamma$ we have 
\[
	h(f|_\Gamma,B\cap \cW^\u_{\rm loc}(x,f))
	= h(f|_\Gamma,B) .	
\]
\end{proposition}

Now let us briefly sketch our strategy to prove  our theorems: 
(i) In Section~\ref{sec:horse} we consider so-called approximating \emph{$(\mu,\varepsilon)$-horseshoes} which -- in entropy and dimension --  approximate a hyperbolic ergodic  measure.
This will enable us to reduce in a way the proof of Theorems~\ref{the:smallsmall} and~\ref{main} for a set $W$ which carries a hyperbolic ergodic measure with positive entropy to the proof of Theorem~\ref{the:basic} for a basic set $W$.
(ii) The local product structure of basic sets allows to reduce the analysis of a limit exceptional set to the analysis of its intersection with unstable manifolds. (iii) Since the exceptional set is $\s$-saturated and invariant (Lemma~\ref{lem:locally}) we can conclude that the entropy on unstable manifolds is equal to the entropy of the full basic set (Proposition~\ref{pro:Manning}). (iv) By approximating almost homogeneous horseshoes (defined in Section~\ref{sec:horse}) we can conclude about dimension (Proposition~\ref{pro:localdim}). (v) Finally, the fact that the exceptional set is $\s$-saturated (Lemma~\ref{lem:locally}) and a slicing argument by Marstrand will help to derive an estimate of the dimension of subsets of direct products of sets (Lemma~\ref{lem:Marstrand}).

 In Section~\ref{sec:entropies} we recall the definition of entropy and prove Theorem~\ref{teo:mainentropy}.  Approximating horseshoe basic sets will enable to conclude  Theorems~\ref{the:basic} and~\ref{the:anosov} and hence Theorems~\ref{the:smallsmall} and~\ref{main}, see Section~\ref{sec:final}. 

\section{Dimensions}\label{sec:dis}

We collect some definitions and standard results on dimension of hyperbolic sets and measures (see also~\cite{Fal:85,Fal:03,Pes:97})  and discuss some examples.
\subsection{Hausdorff dimension}

Let $(X,d)$ be a metric space. Following the general approach of defining Hausdorff dimension in~\cite{Pes:97}, consider a family $\cF$ of subsets of $X$ satisfying the following properties: 
\begin{enumerate}
\item[(HD1)] We have $\emptyset\in\cF$ and $\diam U>0$ for every nonempty $U\in\cF$. 
\item[(HD2)] For every $\varepsilon>0$ there exists a finite or countable subcollection $\cF'\subset\cF$ such that $\bigcup_{U\in\cF'}U\supset X$ and $\diam U\le\varepsilon$ for every $U\in\cF'$. 
\item[(HD3)] There exist positive constants $c_1,c_2$ such that every $U\in\cF$  contains an open set of diameter $c_1\diam U$ and is contained in an open set of diameter $c_2\diam U$. 
\end{enumerate}

Given a set $Y\subset X$ and a nonnegative number $d \in\bR$, we denote the {\it $d$-dimensional Hausdorff measure} of $Y$ (relative to the family $\cF$) by   
\[
	\cH^d(Y)
	\eqdef \lim_{r \to 0}\cH_{r}^d(Y),
	\text{ where }
	\cH^d_r(Y)
	\eqdef \inf\left\{\displaystyle\sum_{i=1}^{\infty}r(U_i)^d\colon  
		Y \subset \bigcup_{i=1}^\infty U_i, r( U_i) <r\right\} ,
\]
where $r( U_i)$ denotes the diameter of $U_i$. Observe that $\cH^d(Y)$ is monotone  nonincreasing in $d$. 
Furthermore, if $d\in(a,b)$ and $\cH^d(Y)<\infty$ 
then $\cH^b(Y)=0$ 
and $\cH^a(Y)=\infty.$
The unique value $d_0$ at which $d\mapsto \cH^d(Y)$ 
jumps from $\infty$ to $0$ is the {\it Hausdorff dimension} of $Y$, that is,  
\[
	\dim_\H  (Y)
	\eqdef \inf\{d\geq 0 \colon \cH^d(Y)=0\}
	= \sup\{d\geq 0 \colon \cH^d(Y)=\infty\}.
\]

Note that the classical definition considers as $\cF$ the family of open sets. Note that the Hausdorff dimension of a set does not depend on the family $\cF$, though the value of the Hausdorff measures may be different (see~\cite[Chapter~1.1]{Pes:97}).

We recall some properties:
\begin{itemize}
\item [(H1)] Hausdorff dimension is monotone: if $Y_1\subset Y_2\subset X$ then $\dim_\H ( Y_1)\leq\dim_\H ( Y_2)$. 
\item [(H2)] Hausdorff dimension is countably stable: $\dim_\H (\bigcup_{i=1}^\infty B_i)=\sup_i\dim_\H (B_i)$. 
\item[(H3)] Hausdorff dimension is bi-Lipschitz invariant: If $f\colon X\longrightarrow X$ is bi-Lipschitz, then $\dim_\H( Y) = \dim_\H( f(Y))$ for all $Y\subset X$.
\end{itemize}

 Below we will use the following crucial property of the Hausdorff dimension of subsets of product sets (similar arguments were used in~\cite[1.4]{KleWei:96}).
	 
\begin{lemma}\label{lem:Marstrand}
	Let $B_1,B_2$ be two metric spaces and let $C$ be some subset of the direct product $B_1\times B_2$. If there are numbers $b_1,b_2$ such that   
\[
	\dim_\H( B_1)\ge b_1
	\quad\text{ and }
	\quad 
	\dim_\H (C\cap (\{y\}\times B_2))\ge b_2 \,\,\text{  for every }y\in B_1
\]	
then
\[
	\dim_\H( C)\ge b_1+b_2.
\]
\end{lemma}	 

\begin{proof}
Let $t<b_2$. Observe that by~\cite[Theorem 5.6]{Fal:85}, $B_1$ contains some subset $B_1'$ of positive $s$-dimensional Hausdorff measure for any $s<b_1$. Hence, applying Marstrand's Theorem (see, for example,~\cite[Theorem 5.8]{Fal:85}	) we have $\dim_\H( C)\ge s+t$. As $s<b_1$ and $t<b_2$ were arbitrary, the claimed property follows.
\end{proof}

\subsection{Dimension of basic sets}

We denote by $\cM(g|_\Gamma)$  the space of $g$-invariant Borel probability measures supported on $\Gamma$ and by $\cM_{\rm erg}(g|_\Gamma)$ the subset of ergodic measures. Given a continuous function $\phi\colon\Gamma\to\bR$, denote by $P_{g|\Gamma}(\phi)$ its \emph{topological pressure} (with respect to $g|_\Gamma$). Recall that it satisfies the \emph{variational principle}
\[
	P_{g|\Gamma}(\phi)
	= \max_{\mu\in\cM_{\rm erg}(g|\Gamma)}\big(h_\mu(g)+\int\phi\,d\mu\big)
\]
(see~\cite{Wal:81} for the definition of pressure and its properties).
Consider the functions $\varphi^\s,\varphi^\u\colon\Gamma\to\bR$ defined by
\begin{equation}\label{eq:potfun}
	\varphi^\s(x)\eqdef \log\,\lVert dg|_{E^\s_x}\rVert,\quad
	\varphi^\u(x)\eqdef - \log\,\lVert dg|_{E^\u_x}\rVert.
\end{equation} 
Let $d^\u$ and $d^\s$ be the unique real numbers for which we have
\begin{equation}\label{eq:localdimension}
	P_{g|\Gamma}(d^\u\varphi^\u)=0 = P_{g|\Gamma}(d^\s\varphi^\s).
\end{equation}
Note that
\[
	d^\s(\Gamma)
	= \max_{\mu\in\cM_{\rm erg}(g|\Gamma)}\frac{h_\mu(g)}{\lvert\chi^\s(\mu)\rvert},\quad
	d^\u(\Gamma)
	= \max_{\mu\in\cM_{\rm erg}(g|\Gamma)}\frac{h_\mu(g)}{\chi^\u(\mu)}.
\]
By classical results (see {\cite{Man:81,McCMan:83,Tak:88, PalVia:88}}), for every $x\in\Gamma$ we have
\begin{equation}\label{eq:stunstdim}
	d^\s(\Gamma)\eqdef \dim_\H (\Gamma\cap \cW^\s_{\rm loc}(x,g)) ,\quad
	d^\u(\Gamma)\eqdef \dim_\H (\Gamma\cap \cW^\u_{\rm loc}(x,g)) 
\end{equation}
and
\begin{equation}\label{eq:summ}
	\dim_\H(  \Gamma) 
	= d^\s(\Gamma)+d^\u(\Gamma).
\end{equation}

\subsection{Dynamical and hyperbolic dimensions}\label{sec:hyp}

The dynamical dimension defined in~\eqref{def:DD} is in a hyperbolic (like) context strongly related with other dimension characteristics. 
Define the \emph{hyperbolic dimension} of $f|_{W}$ by
\[
	\hD(f|_{W})
	\eqdef \sup_Y\dim_\H( Y),
\]
where the supremum is taken over all basic sets $Y\subset W$. 

We call a measure $\mu$ at which the supremum in~\eqref{def:DD} is attained a \emph{measure of maximal dimension} (with respect to $f|_{ W}$). For every basic set $\Gamma\subset M$ by~\cite{BarWol:03} there exists an ergodic $f$-invariant probability measure of maximal dimension (with respect to $f|_\Gamma$). Though, in general such measure is not unique and in general it is not a measure of \emph{full} dimension $\dim_\H(\Gamma)$. Indeed, the formula~\eqref{eq:summ} involving two -- in general independent -- maxima indicates that these facts depend on cohomology relations of the potential functions~\eqref{eq:potfun} (see~\cite{Wol:06} for more details).

While in the case of a rational map the (analogously defined) dynamical dimension and the hyperbolic dimension coincide%
\footnote{For $J\subset\overline\bC$ being the Julia set of a general rational function $f$ of degree $\ge2$ of the Riemann sphere we have $\DD(f|_J)=\hD(f|_J)$ (see~\cite[Chapter 12.3]{PrzUrb:10}), and $\hD(f|_J)=\dim_\H( J)$ if $f|_J$ is expansive (see~\cite[Theorem 3.4]{Urb:03}),  there are examples with $\hD(f|_J)<\dim_\H( J)=2$ (see~\cite{AviLyu:}).}
this is not true in general in our setting. Note that Lemma~\ref{lemaqeaprox} below implies  the first of the following inequalities (the second one is trivial)
\[
	\DD(f|_W)
	\le \hD(f|_W)
	\le \dim_\H( W).
\]
The last inequality can be strict (consider, for example, Bowen's figure-8 attractor). 
The first inequality can also be strict even if $f|_W$ is hyperbolic as we explain in the following examples.

\begin{example}\label{ex:1}
Rams~\cite{Ram:05} provides an example of an affine horseshoe $W\subset\bR^2$ with exactly two measures $\mu_1,\mu_2$ of maximal dimension such that
\[
	\dim_\H\mu_1
	= \dim_\H\mu_2
	= \DD(f|_W)
	< \dim_\H( W). 
\]
Hence, none of those measures can coincide with the (unique) equilibrium state for the potentials $d^\s(W)\varphi^\s$ and $d^\u(W)\varphi^\u$, respectively.
Note that  in this example we clearly have $\hD(f|_W)=\dim_\H( W)$.
\end{example}

To investigate the hyperbolic dimension, we will now consider  Anosov diffeomorphisms which are \emph{volume preserving}, that is, which admit an invariant measure which is absolutely continuous with respect to the one induced by the Riemannian metric  $m$.  Recall that a $C^2$ Anosov diffeomorphism $g$ is volume preserving if, and only if, $g$ admits an invariant measure of the form $d\mu=h\,dm$ with $h$ a positive H\"older continuous function if, and only if, $dg^n\colon T_xM\to T_xM$ has determinant $1$ whenever $g^n(x)=x$ if, and only if, $\varphi^\s$ is cohomologous
 to $-\varphi^\u$.%
\footnote{The first equivalence follows from~\cite[4.14 Theorem]{Bow:75}. For the second one recall that two functions $\varphi$ and $\psi$ are \emph{cohomologous} if $\varphi-\psi=\eta-\eta\circ f$ for some continuous function $\eta$. Note that if we have $g^n(x)=x$ and $\lvert\det dg^n_x\rvert=1$ then for every $\ell\in\bZ$ we have $1= \lvert\det dg^{\ell n}_x\rvert=\exp(\ell\varphi^\s(g^n(x))+\ell\varphi^\u(g^n(x)))\sin(\angle(E^\s_x,E^\u_x))$. Letting $\ell\to\pm\infty$ we conclude $\varphi^\s(g^n(x))+\varphi^\u(g^n(x))=0$. Thus, by Livshitz's theorem $\varphi^\s$ is cohomologous to $-\varphi^\u$ (see~\cite[Theorem 19.2.1]{KatHas:95}).} 
By~\cite[4.15 Corollary]{Bow:75}, among the $C^2$ Anosov diffeomorphisms the ones which are not volume preserving form an open and dense subset.

\begin{example}\label{ex:2}
Let $f$ be a  $C^2$  Anosov (hence mixing) diffeomorphism of $\bT^2$. Recall that there exists a unique SRB measure $\mu_{\rm SRB}^+$ (with respect to $f$). By~\eqref{eq:Young} and~\eqref{pesfor} we have
\[
	\dim_\H\mu_{\rm SRB}^+
	= 1 - \frac{\chi^\u(\mu_{\rm SRB}^+)}{\chi^\s(\mu_{\rm SRB}^+)}. 
\]
Note that there exists also a unique SRB measure $\mu_{\rm SRB}^-$ (with respect to $f^{-1}$) which has analogous properties. Note that by~\cite{LedYou:85} the SRB measures (with respect to $f^{\pm1}$, respectively) are the only ergodic measures satisfying the equality~\eqref{pesfor}, that is, for every ergodic $\mu\ne\mu_{\rm SRB}^\pm$ we have $\chi^\s(\mu)<-h(\mu)\le0\le h_\mu(f)<\chi^\u(\mu)$. 
 
It is hence an immediate consequence that an ergodic measure $\mu$ of \emph{full} dimension $\dim_\H\mu=2$ satisfies $h_\mu(f)/\chi^\u(\mu)=1=-h_\mu(f)/\chi^\s(\mu)$ and that hence $\mu=\mu_{\rm SRB}^+=\mu_{\rm SRB}^-$. In particular, such measure is unique. Moreover, by~\cite{You:82} for $\mu$-almost every $x$ we have
\[
	\lim_{\varepsilon\to0}\frac{\log\mu(B(x,\varepsilon))}{\log\varepsilon}
	= \frac{h_\mu(f)}{\chi^\u(\mu)}-\frac{h_\mu(f)}{\chi^\s(\mu)}
	= 2.
\]
Hence, $\mu$ is absolutely continuous with respect to the $2$-dimensional Hausdorff measure which is positive and finite (see by~\cite[Theorem 4.3.3]{BisPer:16}).

On the other hand, if $f$ preserves a measure $\mu$ which is absolutely continuous with respect to the $m$, then it has absolutely continuous measure on unstable and stable manifolds, respectively, that is, $\mu=\mu_{\rm SRB}^+=\mu_{\rm SRB}^-$. Hence $\dim_\H \mu=2$. 

In particular, this implies that in the case if $f$ is not  volume preserving then
\[
	\dim_\H\mu^\pm_{\rm SRB}
	\le \DD(f|_{\bT^2})
	< 2
	= \hD(f|_{\bT^2})
	= \dim_\H(\bT^2).
\]

By~\cite[Theorem 3]{Dol:97}, for any set $A\subset\bT^2$ satisfying $\dim_\H( A)< D\eqdef\dim_\H\mu_{\rm SRB}^+$ we have $\dim_\H( I^+_{f|\bT^2}(A))=2$.  The hypothesis on $A$  is optimal in the sense that by~\cite{Dol:97} in the case we have $D<2$ then for every $s\in(D,2]$ there exists a set $A\subset\bT^2$ such that the $A$-exceptional set has Hausdorff dimension $<2$.

Note that in general 
\[
	\dim_\H\mu_{\rm SRB}^+
	\le \DD(f|_{\bT^2}).
\]
In the case when the SRB measure is not a measure of maximal dimension then for every  measure $\mu\in\cM_{\rm erg}(f|_{\bT^2})$ satisfying  
\[
	\dim_\H\mu>\dim_\H\mu^+_{\rm SRB}
\]	 
for every set $A\subset\bT^2$ such that $\dim_\H\mu^+_{\rm SRB}<\dim_\H( A)<\dim_\H\mu$ by Theorem~\ref{the:anosov} we have
\[
	1+\frac{h_\mu(f)}{\chi^\u(\mu)}
	\le \dim_\H( I^+_{f|\bT^2}(A))
	\le 2.
\]
\end{example}

\section{Approximating horseshoes}\label{sec:horse}
 
A point $x\in M$ is \emph{Lyapunov regular} and \emph{hyperbolic} if there exist numbers $\chi^\s(x)<0<\chi^\u(x)$ and a decomposition $ E^\s_x\oplus E^{u}_x=T_xM$ into subspaces of dimension $1$ such that for $\star=\s,\u$  we have
\[
 \chi^\star(x)=\lim_{|k|\to\infty}\frac{1}{k}\log\,\lVert df^k_x(v)\rVert
\]
whenever $v\in E^\star_x\setminus\{0\}$ where $\star=s,u$. 
We call an ergodic $f$-invariant Borel probability  measure $\mu$ \emph{hyperbolic} if $\mu$-almost every point is Lyapunov regular hyperbolic and for such $\mu$ denote
\[
	\chi(\mu)
	\eqdef \min\{\lvert\chi^\s(\mu)\rvert,\chi^\u(\mu)\}.
\]

\begin{definition}
	Given numbers $\chi^\s<0<\chi^\u$ and $\varepsilon\in(0,\min\{\lvert\chi^\s\rvert,\chi^\u\})$, we call a basic set $\Gamma\subset M$ 
a \emph{$(\chi^\s,\chi^\u,\varepsilon)$-horseshoe} if for every $x\in\Gamma$ we have
\[
	\limsup_{\lvert n\rvert\to\infty}\Big\lvert \frac1n\log\,\lVert df^n|_{E^\s_x}\rVert-\chi^\s\Big\rvert<\varepsilon,
	\quad
	\limsup_{\lvert n\rvert\to\infty}\Big\lvert \frac1n\log\,\lVert df^n|_{E^\u_x}\rVert-\chi^\u\Big\rvert<\varepsilon.
\]
Given an ergodic hyperbolic measure $\mu$ and $\varepsilon\in(0,\chi(\mu))$, we call a basic set $\Gamma\subset M$ a \emph{$(\mu,\varepsilon)$-horseshoe} (with respect to $f$) if it is a $(\chi^\s(\mu),\chi^\u(\mu),\varepsilon)$-horseshoe, if it is in an $\varepsilon$-neighborhood of $\supp\mu$, and if $\lvert h(f|_\Gamma)-h_\mu(f)\rvert<\varepsilon$.
\end{definition}

The existence of such horseshoes, and hence the proofs of the following two lemmas, follows from Katok's construction (see~\cite[Supplement S.5]{KatHas:95}, see also~\cite{Gel:16}).

\begin{lemma}\label{lemaqeaprox1}
	Given a hyperbolic ergodic measure $\mu\in\cM$, there exists a function $\delta\colon(0,1]\to\bR$ satisfying $\delta(\varepsilon)\to0$ as $\varepsilon\to0$ such that for every $\varepsilon>0$ there exist a positive integer $N$ and $\Gamma=\Gamma(\mu,\varepsilon)\subset W$ a $(\mu,\varepsilon)$-horseshoe (with respect to $f^N$) such that
\[
	h(f|_W,\Gamma)\ge h_\mu(f)-  \varepsilon
	\quad\text{ and }\quad
	d^\star(\Gamma)
	\ge \frac{h_\mu(f)}{\lvert \chi^\star(\mu)\rvert}
	-\delta(\varepsilon),
\]	
for $\star=\s,\u$ respectively. Moreover, there is $R\subset\Gamma$ such that $f^N|_R$ is conjugate to a full shift and $\Gamma=\bigcup_{i=1}^Nf^i(R)$.
\end{lemma}

\begin{proof}
There exists a $(\mu, \varepsilon)$-horseshoe $\Gamma$ and  a positive integer $N=N(\varepsilon)$ and $R\subset \Gamma$ such that  
$\Gamma = \bigcup_{i=0}^{N}f^i(R)$, $f^N|_R$ is hyperbolic and conjugate to a (mixing) full shift (see, for example \cite{KatHas:95} or \cite[Theorem 1]{Gel:16}). In particular, we have  $\dim_\H( \Gamma)=\dim_\H( R)$. 
Our assumption that $W$ is locally maximal guarantees that $\Gamma\subset W$ if $\varepsilon$ is sufficiently small.

Applying~\eqref{eq:localdimension} to $f^N|_R$, with $d^\u(\Gamma)=d^\u(R)=\dim_\H(R\cap\cW^\u_{\rm loc}(x,f^N))$ we have
\[
	0
	= \sup_{\nu\in\cM_{\rm erg}(f^N|R)}\Big(h_\nu(f^N) - N d^\u(R)\chi^\u(\nu)\Big).
\]
Recall that for every ergodic measure $\nu$ for $f^N\colon R\to R$ we get an invariant measure $\hat\nu$ for $f\colon \Gamma\to \Gamma$ by defining $\hat\nu\eqdef \frac1N(\nu+f_\ast\nu+\ldots+f^{N-1}_\ast\nu)$ and observe that $h_\nu(f^N)=Nh_{\hat\nu}(f)$ and $\chi^u(\nu)= N \chi^u(\hat{\nu})$. Further, $h(f^N|_R)=Nh(f|_\Gamma)$. 
By the variational principle for topological entropy (see (E6) in Section~\ref{sec:entropies}), we can take $\nu$ such that $h_\nu(f^N) \ge Nh(f|_\Gamma) -N\varepsilon$, which implies 
\[
	0
	\ge h(f|_\Gamma) - \varepsilon - d^\u(R)\chi^\u(\nu). 
\]
By the defining properties of the $(\chi^\s(\mu),\chi^\u(\mu), \varepsilon)$-horseshoe  we have 
\[
	0 
	\ge  h_\mu(f) -2\varepsilon- d^\u(\Gamma) (\chi^\u(\mu)+\varepsilon),
\] 
which implies
\[
	d^\u(\Gamma)
	\ge \frac{h_\mu(f) -2\varepsilon}{\chi^\u(\mu)+\varepsilon}.
\]
Analogously, with $d^\s(\Gamma)=\dim(R\cap\cW^\s_{\rm loc}(x,f^N))$ we have
\[
	d^\s(\Gamma)
	\ge \frac{h_\mu(f) -2\varepsilon}{-\chi^\s(\mu)+\varepsilon}.
\]
Now~\eqref{eq:summ} implies the claimed properties.
\end{proof}

\begin{lemma}\label{lemaqeaprox} 
If $\DD(f|_W)>0$ then there exist a sequence of hyperbolic ergodic measures $(\mu_n)_n\subset\cM$, a sequence of positive numbers $(\varepsilon_n)_n$ with $\lim_{n\to 0}\varepsilon_n = 0$, and a sequence of $(\mu_n, \varepsilon_n)$-horseshoes   $\Gamma_n = \Gamma_n(\mu_n,\varepsilon_n)\subset W$ satisfying  
\[
	\lim_{n\to \infty}  \dim_\H(  \Gamma_n )
	= \lim_{n\to\infty} \dim_\H \mu_n
	= \DD(f|_W).
\]
\end{lemma}

\begin{proof}
	To prove the lemma it suffices to observe that $\DD(f|_W)>0$ implies that every ergodic $\mu\in\cM$ for which $\dim_\H\mu$ is sufficiently close to $\DD(f|_W)$ is hyperbolic. Now take a sequence $(\mu_n)_n\subset\cM$ of such measures such that $\dim_\H \mu_n\to\DD(f|_W)$ and apply Lemma~\ref{lemaqeaprox1}.
\end{proof}

\section{Entropies}\label{sec:entropies}

We briefly recall the definition of entropy according to Bowen~\cite{Bow:73}.
Let $X$ be a compact metric space. Consider a continuous map $f\colon X\to X$, a set $Y\subset X$,  and a finite open cover $\mathscr{A} = \{A_1, A_2,\ldots, A_n\}$ of $X$. Given $U\subset X$ we write $U \prec \mathscr{A}$ if there is an index $j$ so that $U\subset A_j$, and $U\nprec\mathscr{A}$ otherwise. 
Given $U\subset X$ we define 
\[
	n_{f,\mathscr{A}}(U) 
	\eqdef
		\begin{cases}
		0&\text{ if } U \nprec \mathscr{A},\\
		\infty &\text{ if } f^k(U)\prec \mathscr{A}\,\,\forall k\in\mathbb{N},\\
		 \ell&\text{ if }  f^k(U)\prec \mathscr{A}\,\, \forall k\in \{0, \dots,  \ell-1\},f^\ell(U)\nprec\mathcal{A}.
		\end{cases}
\]
If $\mathcal U$ is a countable collection of open sets, given $d>0$ let
\[
	 m(\mathscr A,d,\mathcal U)
	\eqdef \sum_{U\in\mathcal U}e^{-d \,n_{f,\mathscr{A}}(U)}.
\]
Given a set $Y\subset X$, let 
\[
	m_{\mathscr{A}, d} (Y) 
	\eqdef \lim_{\rho \to 0}\inf \Big\{m(\mathscr A,d,\mathcal U)\colon
		Y \subset\bigcup_{U\in\mathcal U} U, e^{-n_{f,\mathcal{A}}(U)}<\rho
		\text{ for every } U\in\mathcal U
	\Big\}.
\]
Analogously to the Hausdorff measure, $d\mapsto m_{\mathcal{A},d}(Y)$ 
jumps from $\infty$ to $0$ at a unique critical point and one defines
\[
  h_{\mathscr{A}}(f,Y) 
	\eqdef \inf\{d\colon m_{\mathscr{A}, d}(Y)=0\}
   = \sup\{d\colon m_{\mathscr{A}, d}(Y)=\infty\}. 
\]
The \emph{topological entropy} of $f$ on $Y$ is defined by 
\[
	h(f,Y) 
	\eqdef \sup_{\mathscr{A}} h_{\mathscr{A}}(f,Y) .
\]
Note that $Y$ does not need to be compact nor invariant.
When $Y=X$, we simply write $h(f) = h(f,X)$ when there is no risk of confusion. To point out the (sub)space we consider, we sometimes write $h(f|_X,Y)$. 
In the case of a compact set $Y$ this definition is equivalent to the canonical definition of topological entropy (see~\cite[Proposition 1]{Bow:73} and \cite[Chapter 7]{Wal:81}). 

We recall some properties which are relevant in our context (see~\cite{Bow:73,Pes:97}). 
\begin{itemize}
\item[(E1)] If $f\colon X\to X$ and $g\colon Y \to Y$ are topologically semi-conjugate, that is, there is a continuous map $\pi\colon  X \to Y$ with $g\circ \pi=\pi \circ f$, then $h(g,\pi(A))\le h(f,A)$ for every $A\subset X$. 
\item[(E2)] Entropy is invariant under iteration: $h(f,f(A)) =  h(f,A)$ for every $A\subset X$. 
\item[(E3)] Entropy is countably stable: $h(f,\bigcup_{i=1}^\infty A_i ) = \sup_i h(f,A_i).$
\item[(E4)]  $h(f^m,A) = m\cdot h(f,A)$ for all $m\in\bN$ for every $A\subset X$.
 \item[(E5)] Entropy is monotone: if $A\subset B\subset X$ then $h(f,A)\le h(f,B)$.
\item[(E6)] Variational principle: $h(f)=\sup_{\mu\in\cM(f)} h_\mu(f)$. 
\end{itemize}

We recall two technical results. Given a positive integer $M$, let $\sigma^+\colon\Sigma_M^+\to\Sigma_M^+$ be the usual one-sided shift map on $\Sigma_M^+=\{1,\ldots,M\}^{\bN_0}$ and $\sigma\colon\Sigma_M\to\Sigma_M$  the  one on $\Sigma_M=\{1,\ldots,M\}^\bZ$.

\begin{theorem}[{\cite[Theorem 1]{Dol:97}}]\label{pro1:Dol}
	If $A\subset\Sigma_M^+$ satisfies $h(\sigma^+,A)<h(\sigma^+)$, then we have $h(\sigma^+,I^+_{\sigma^+|\Sigma_M^+}(A))=h(\sigma^+)$.	
\end{theorem}

Given $n\ge1$, let $\Sigma_{M,n}^+=\{1,\ldots,M\}^n$. Denote by $\lvert U\rvert\eqdef n$ the length of $U\in\Sigma_{M,n}^+$.

\begin{proposition}[{\cite[Section 3, Proposition 1]{Dol:97}}]\label{pro:Dol}
	Given $\cU\subset\bigcup_n\Sigma_{M,n}^+$ denote
\[
	I^+(\cU)
	\eqdef \big\{\underline i\in\Sigma_M^+\colon \forall n<m\text{ we have }(i_n\ldots i_m)\not\in\cU\big\}.
\]	

Then there exists a function $H\colon\bN\to\bR$ having the property
\[
	\lim_{n\to\infty}H(n)= h(\sigma^+)
\]
so that if for $s\in(0,h(\sigma^+|_{\Sigma_M^+}))$ and  $n_0\ge1$ there is a family $\cU=\{ U_\ell\colon U_\ell \in\Sigma_{M,n}^+,n\ge n_0\}$ satisfying
\[
	\sum_\ell e^{-s\lvert U_\ell\rvert}<1,
\]	 
then we have $h(\sigma^+,I^+(\cU))\ge H(n_0)$. 
\end{proposition}

The above listed properties of entropy and basic properties of exceptional sets immediately imply the following result (see for example~\cite[Sections 4--5]{CamGel:16}). 

\begin{lemma}\label{lem:iterationentropy}
	Let $f\colon X\to X$ be a homeomorphism and $R\subset\Gamma\subset X$  sets such that $\Gamma=\bigcup_{i=1}^Nf^i(R)$ for some positive integer $N$ and such that $f^N|_R$ is conjugate to a full shift $\sigma\colon\Sigma_M\to\Sigma_M$. Then with $g=f^N$ we have
\[
	Nh(f|_\Gamma,I^+_{f|\Gamma}(A\cap\Gamma))
	= h(g|_R,I^+_{g|R}(A\cap R)).
\]	
\end{lemma}
\begin{proof}
	It suffices to observe that $\bigcup_{i=1}^Nf^i(I^+_{g|R}(A\cap R))=I^+_{f|\Gamma}(A\cap\Gamma)$ and to apply (E2) and (E4). 
\end{proof}

The following is an immediate consequence of continuity. 

\begin{lemma}\label{lempropriE1}
If $f\colon X \to X$ and $g\colon Y \to Y$ are topologically semi-conjugate by a continuous map $\pi \colon X \to Y$ with $\pi \circ f = g\circ \pi$, then for every $A \subset X$ we have  
\[
	I^+_{g|Y}(\pi A)
	\subset \pi I^+_{f|X}(A) . 
\]
\end{lemma}

We can now give the proof of one of our main results.

\begin{proof}[Proof of Theorem~\ref{teo:mainentropy}]
By hypothesis, we have $h(f|_W)>0$.
By the variational principle for entropy (E6) and Ruelle's inequality for entropy, for every $\varepsilon>0$ there is a hyperbolic ergodic measure $\mu\in\cM$ satisfying $h_\mu(f) \ge h(f|_W)  - \varepsilon$. 
By Lemma~\ref{lemaqeaprox1} there are a positive integer $N=N(\varepsilon)$ and a basic set $\Gamma_\varepsilon\subset W$ (with respect to $f^N$) such that
\[
	h(f|_{\Gamma_\varepsilon})\ge h_\mu(f) - \varepsilon.
\]	 
If $\varepsilon$ was sufficiently small, this and our hypothesis  $h(f|_W, A) < h(f|_W)$ together imply  $h(f|_{\Gamma_\varepsilon}, A\cap \Gamma_\varepsilon) < h(f|_{\Gamma_\varepsilon})$.

By (E4) and Lemma~\ref{lem:iterationentropy}, without loss of generality we can assume that $N=1$ and that $f|_{\Gamma_\varepsilon}$ is conjugate  to a mixing (two-side) full shift
$\sigma\colon\Sigma_M\to\Sigma_M$,  for some positive integer $M=M(\Gamma_\varepsilon)$, by means of a homeomorphism $p\colon \Gamma_\varepsilon\to\Sigma_M$ satisfying $p\circ f=\sigma\circ p$.  
Denote by $\pi^+\colon\Sigma_M\to\Sigma^+_M$ the natural projection $\pi^+(\ldots i_{-1}i_0i_1\ldots)=(i_0i_1\ldots)$. 
Note that $h(f|_{\Gamma_\varepsilon})=h(\sigma|_{\Sigma_M})=h(\sigma^+|_{\Sigma_M^+})$.
By (E1) applied to $\pi^+\circ p$ we have
\[
	h(\sigma^+,(\pi^+\circ p)(A))
	\le h(\sigma,p(A))
	= h(f|_{\Gamma_\varepsilon},A)
	< h(f|_{\Gamma_\varepsilon})  
	=h(\sigma)
	=h(\sigma^+).
\]
Hence, by Theorem~\ref{pro1:Dol}, we have
\[
	h\big(\sigma^+,I^+_{\sigma^+|\Sigma^+}((\pi^+\circ p)(A))\big)
	= h(\sigma^+).
\]
Lemma~\ref{lempropriE1} implies $(\pi^+\circ p)(I^+_{f|\Gamma_\varepsilon}(A))\supset I^+_{\sigma^+|\Sigma^+}((\pi^+\circ p)(A))$ and hence that
\[
	h\big(\sigma^+,(\pi^+\circ p)(I^+_{f|\Gamma_\varepsilon}(A))\big)
	= h(\sigma^+).
\]
Hence, (E1) implies
\[
	h(f|_{\Gamma_\varepsilon},I^+_{f|\Gamma_\varepsilon}(A))
	= h(\sigma^+) 
	= h(f|_{\Gamma_\varepsilon})
	\ge h_\mu(f) - \varepsilon.
\]
Now apply (E5) to $I^+_{f|\Gamma_\varepsilon}(A)\subset I^+_{f|W}(A)$. Since $\varepsilon$ was arbitrary, this implies $ h(f|_W)=h(f|_W, I^+_{f|W}(A))$.

By Remark \ref{remarkIE} and (E3) we have 
\[
	h(f|_W) 
	= h(f|_W, I^+_{f|W}(A)) 
	= \sup_{n\ge0}\Big\{ h(f|_W, E^+_{f|W}(A)), 
	 h\big(f|_W, f^{-n}(\tilde{A})\big)\Big\}.
\] 
By (E2), (E5), and our hypothesis we have 
\[
	h(f|_W, f^{-n}\big(\tilde{A}\big)) = h(f|_W, \tilde{A})\leq h(f|_W, A) < h(f|_W).
\]	
Hence, with the above, we have $h(f|_W, I^+_{f|W}(A))= h(f|_W, E^+_{f|W}(A)) = h(f|_W)$. 
This proves the theorem.
\end{proof}

\begin{proof}[Proof of Proposition~\ref{pro:Manning}]
By (E5), $h(f|\Gamma,B)\ge h(f|\Gamma,B\cap \cW^\u_{\rm loc}(x,f))$.
It remains to show the other inequality. To sketch the proof recall that, by definition of a $\s$-saturated set, if a point is in  $B$ then so is any point in its local stable manifold. That is, we can express $B$ as a union of subsets of local stable manifolds. We will intersect this set by the local unstable manifold through $x$. This will enable us to pass from a cover of $B\cap \cW^\u_{\rm loc}(x,f)$ to a cover of $B$ to estimate entropy. 
 
By (E2)  it suffices to show the equality for $f^m$ instead for $f$. Also note that for $\varepsilon>0$ sufficiently small, by hyperbolicity of $\Gamma$ for $m\ge1$ sufficiently large the diameter of $f^{km}(\cW^\s_{\rm loc}(x,f))$ is monotonically decreasing in $k$  for every $x\in \Gamma$. Hence, without loss of generality, we can assume that this happens already for $f^k(\cW^\s_{\rm loc}(x,f))$. 

Fix  $x\in \Gamma$ and consider $C\subset \cW^\u_{\rm loc}(x,f)$ any closed curve.
Let $h\eqdef h(f|_\Gamma,B\cap C)$. 
Fix a finite open cover $\cA$ of $\Gamma$.
Let $2\ell$ be a Lebesgue number for $\cA$.

Given $\ell>0$ and $y\in\Gamma$, denote by $\cW^\s_\ell(y,f)$ the intersection of $\cW^\s_{\rm loc}(y,f)$ with an open ball of radius $\ell$ centred in $y$.

Note that, by hyperbolicity of the surface diffeomorphism $f$ on $\Gamma$, for $m\ge1$ sufficiently large we have that
\begin{equation}\label{e:mepsilon}
	f^m(C)\cap \cW^\s_{\ell}(y,f)\ne\emptyset 
	\quad\text{ for every }y\in \Gamma.
\end{equation}

\begin{claim}\label{cla:qwqw}
	For every $y\in B$ there exists $z \in f^m(C)\cap B$ so that  $y\in \cW^\s_{\ell}(z,f)$.
\end{claim}

\begin{proof}
Since $B$ is  $\s$-saturated, $\cW^\s_\ell(y,f)\subset B$. By~\eqref{e:mepsilon} there is  $z \in f^m(C)\cap \cW^\s_{\ell}(y,f)\cap B$. Finally,  note that $y\in \cW^\s_{\ell}(z,f)$. 
\end{proof}

\begin{claim}\label{cla:qw}
	 For every $\varkappa>0$ we have $m_{\cA,h+\varkappa}(B)=0$.
\end{claim}

\begin{proof}
By (E2) and  invariance of $B$, for  $m\ge1$ satisfying~\eqref{e:mepsilon} we have
\[
	h(f|_\Gamma, C\cap B)
	= h\big(f|_\Gamma,f^m(C\cap B)\big)
	= h\big(f|_\Gamma,f^m(C)\cap B\big).
\]
Let $\cB$ be a finite cover of $\Xi$ by open balls of radius $\ell$.
By definition of entropy, $h_\cB(f|_\Gamma,f^m(C)\cap B)\le h(f|_\Gamma,f^m(C)\cap B)=h$ and hence
\[
	m_{\cB,h+\varkappa}(f^m(C)\cap  B)=0.
\]	 
Thus, for any $\delta>0$ we have $m_{\cB,h+\varkappa}(f^m(C)\cap B)<\delta$ and hence there exists $N_0=N_0(\delta)\ge1$ such that for every $N\ge N_0$ there is a countable collection of open sets $\cU$ covering $f^m(C)\cap B$ with $n_{f,\cB}(U)\ge N$ for every $U\in\cU$ and satisfying
\[
	m(\cB,h+\varkappa,\cU)
	= \sum_{U\in\cU}e^{-(h+\varkappa)n_{f,\cB}(U)}<\delta.
\]
The cover $\cU$ of $f^m(C)\cap B$ induces a cover $\cU^\star$ of $f^m(C)\cap B$ by  open (in $W$) sets $U^\star$ where each set is obtained from $U\in\cU$ by setting
\[
	 U^\star
	\eqdef \bigcup_{z\in U}\cW^\s_{\ell}(z,f).
\]
By Claim~\ref{cla:qwqw} for every $y\in B$ there is some $z\in f^m(C)\cap B$ such that $y\in \cW^\s_{\ell}(z,f)$. Hence $\cU^\star$ covers $B$.

To determine the size of the elements of $\cU^\star$ note that by the above choices and observations, for every $U\in\cU$ for every $k\in\{0,\ldots,n_{f,\cB}(U)\}$ we have $f^k(U)<\cB$ and hence $f^k(U)$ has diameter at most $\ell$. Thus, for every $k$ the set $f^k(U^\star)$ has diameter at most $2\ell$ and thus is contained in some element of $\cA$. Hence $n_{f,\cA}( U^\star)\ge n_{f,\cB}(U)\ge N$.
Summarizing, for every $N\ge N_0$ we obtain a cover $\cU^\star$ of $B$ which satisfies $n_{f,\cA}( U^\star)\ge N$ for every $ U^\star\in\cU^\star$ and
\[
	m(\cA,h+\varkappa,\cU^\star)
	= \sum_{ U^\star\in\cU^\star}\exp[-(h+\varkappa)n_{f,\cA}( U^\star)]
	\le \sum_{U\in\cU}\exp[-(h+\varkappa)n_{f,\cB}(U)]<\delta.
\]
Thus, we can conclude $m_{\cA,h+\varkappa}(B)=0$, proving the claim.
\end{proof}

Since $\varkappa>0$ was arbitrary in Claim~\ref{cla:qw}, we obtain $h_\cA(B)\le h$. 
Since $\cA$ was arbitrary, we obtain $h(f|_\Gamma,B)\le h$. Finally, by monotonicity (E1), we have
\[
	h(f|_\Gamma,B)
	\le h(f|_\Gamma,B\cap C)
	\le h(f|_\Gamma,B\cap\cW^\u_{\rm loc}(x,f)).
\]
The proposition is proved.
\end{proof} 

The following result is of similar spirit as~\cite[Lemma 2]{Dol:97}. Its proof is verbatim (hence omitted) to the proof of~\cite[Proposition 2.2]{CamGel:16} (which, in turn,  is inspired by~\cite{Man:81} and \cite[Theorem 1.2]{BurGel:14}). 

\begin{proposition}\label{pro:localdim}
Let $\Gamma\subset M$ be a $(\chi^\s,\chi^\u,\varepsilon)$-horseshoe.  

Then for $B\subset\Gamma$ for every $x\in B$ we have
\[
	\frac{\dim_\H( B\cap\cW^\u_{\rm loc}(x,f))}{\dim_\H(\Gamma\cap\cW^\u_{\rm loc}(x,f))}
	\ge \frac{h(f|_\Gamma,B\cap\cW^\u_{\rm loc}(x,f))}{h(f|_\Gamma)}
		\frac{\chi^\u-\varepsilon}{\chi^\u+\varepsilon}.
\]
\end{proposition}

We need the following  technical result. Its proof follows ideas in~\cite[Section 5.3]{Dol:97}.

\begin{proposition}\label{pro:diment}
	Let $\mu\in\cM_{\rm erg}(f|_W)$ be hyperbolic. 
	Let $A\subset W$ be some set satisfying $\dim_\H( A)<\dim_\H\mu$.	

There exists $\varepsilon_0>0$ such that for every $\varepsilon\in(0,\varepsilon_0)$ and for every set $\Gamma=\Gamma(\varepsilon)\subset W$ which is a $(\mu,\varepsilon)$-horseshoe for $f^N$ for some positive integer $N$ we have
\[
	h(f|_\Gamma,I^+_{f|\Gamma}(A\cap\Gamma))
	= h(f|_\Gamma). 
\]	
\end{proposition}

\begin{proof}
Let $\chi^\mp\eqdef\chi^{\s/\u}(\mu)$.

Note that by Young's formula~\eqref{eq:Young} $\dim_\H \mu>0$ is equivalent to $h_\mu(f)>0$.
Choose $\theta$ such that $\dim_\H( A)<\theta<\dim_\H\mu$. Let
\begin{equation}\label{eq:choosee}
	\delta
	\eqdef 1- \frac{\theta}{\dim_\H\mu} 
\end{equation}
and fix some $\varepsilon_0>0$ small enough such that we have
\begin{equation}\label{eq:chooseee}
	\frac{2\varepsilon_0}{\lvert\chi^-\rvert}<\frac\delta4
	\quad\text{ and }\quad
	\frac{2\varepsilon_0}{h_\mu(f)/2}<\frac\delta8.
\end{equation}

Given now $\varepsilon>0$ such that $\varepsilon<\min\{\varepsilon_0,h_\mu(f)/2\}$, let $\Gamma(\varepsilon)\subset W$ be a $(\chi^-,\chi^+,\varepsilon)$-horseshoe as in Lemma~\ref{lemaqeaprox1} (with respect to $f^N$ for some positive integer $N$). 

Without loss of generality, invoking  property (E4) and Lemma~\ref{lem:iterationentropy}, for the rest of the proof we can assume that $N=1$.

Consider the functions $\psi^-,\psi^+\colon\Gamma\to(-\infty,0)$
\[
	\psi^-(x)\eqdef \log\,\lVert df|_{E^\s_x}\rVert,\quad
	\psi^+(x)\eqdef - \log\,\lVert df|_{E^\u_x}\rVert.
\]	
Because of continuity of  $\psi^-$ and $\psi^+$  and compactness of $\Gamma$, there is a positive constant $C_0$ such that
\begin{equation}\label{eq:boundsss}
	-C_0\le \min\{ \psi^-,  \psi^+\}.
\end{equation}

Observe that by hyperbolicity of $f|_\Gamma$ and the properties of a $(\chi^-,\chi^+,\varepsilon)$-horseshoe, there exists $N_0=N_0(\varepsilon)\ge1$ such that for every $x\in\Gamma$ and for every $n\ge N_0$ we have
\begin{equation}\label{eq:limsupunif}
	\lvert\frac1nS_{-n}\psi^-(x)-\chi^-\rvert<2\varepsilon,\quad
	\lvert\frac1nS_n\psi^+(x)+\chi^+\rvert<2\varepsilon,
\end{equation}
where $S_{-n}\phi=\phi\circ f^{-1}+\phi\circ f^{-2}+\ldots+\phi\circ f^{-n}$ and $S_n\phi=\phi+\phi\circ f+\ldots+\phi\circ f^{n-1}$.

	Let $R_1,\ldots,R_M$ be a Markov partition of $\Gamma$ (with respect to $f$) and recall that $f|_\Gamma$ is topologically conjugate to $\sigma|_{\Sigma_M}$ for some $M\ge1$ by means of some homeomorphism $\pi\colon\Sigma_M\to\Gamma$, $\pi\circ\sigma= f\circ\pi$.

Given $r\in(0,1)$, we construct a Moran cover of pairwise disjoint cylinders of (up to some distortion correction factor) approximately size $r$. First, we consider the potential function $\psi^+$. For every $\xi\in\Sigma_M$ let $n=n(\xi)\ge1$ be the smallest positive integer such that
\[
	S_n\psi^+(\pi(\xi))
	< \log r.
\]

Note that~\eqref{eq:boundsss} implies
\[
	S_n\psi^+(\pi(\xi))
	<  \log r
	\le	S_n\psi^+(\pi(\xi)) + C_0
\]

Since $\psi^+$ is uniformly bounded and negative, there exist positive integers $N_1\le N_2$ depending only on $r$ such that for every $\xi\in\Sigma_M$ we have $N_1\le n(\xi)\le N_2$. 
	We now construct a partition of the associated space of one-sided sequences $\Sigma_M^+=\{1,\ldots,M\}^{\bN_0}$ recursively: Start by setting $m=0$, $S=\Sigma_M^+$, $\cC^+=\emptyset$, and $k=N_1$, and 
\begin{itemize}
\item let $\ell_k$ be the number of (disjoint) cylinders $[\eta_1^i\ldots \eta_k^i]$, $i=1,\ldots,\ell_k$, which contain a sequence $\eta^i\in\Sigma_M^+$ with $n(\eta^i)=k$;
\item replace $\cC^+$ by $\cC^+\cup\bigcup_{i=1}^{\ell_k}[\eta_1^i\ldots\eta_k^i]$, and replace $S$ by $S\setminus\{[\eta_1^i\ldots\eta_k^i]\colon i=1,\ldots,\ell_k\}$;
\item if $k=N_2$ or $S=\emptyset$ then stop the recursion. Otherwise,  repeat the recursion replacing $k$ by $k+1$. 
\end{itemize}
Since $n(\cdot)\le N_2$, the recursion eventually stops with $\cC^+=\Sigma_M^+$. The thus obtained family $\cC^+$ provides a partition of $\Sigma_M^+$ which has the following properties: 
\begin{itemize}
\item It is a family of (pairwise disjoint) cylinders which each are of level between $N_1$ and $N_2$.
\item Each cylinder of level $k$ contains a sequence $\xi\in\Sigma_M^+$ with $n(\xi)=k$ and any sequence $\eta\in[\xi_1\ldots\xi_k]$ satisfies $n(\eta)\in\{k,\ldots,N_2\}$.
\end{itemize}
We call $\cC^+(r)$ a \emph{Moran cover} of $\Sigma_M^+$ of parameter $r$ (relative to the function $\psi^+$).
We will below also keep track of the corresponding positive integer $N_1$ which we hence denote by $N_1^+(r)$.

Now we consider the potential function $\psi^-$. For every $\xi\in\Sigma_M$ let $n=n(\xi)\ge1$ be the smallest positive integer such that
\[
	S_{-n}\psi^-(\pi(\xi))
	< \log r.
\]
Note that~\eqref{eq:boundsss} implies
\begin{equation}\label{eq:bounds-s}
	S_{-n}\psi^-(\pi(\xi))
	<  \log r
	\le	S_{-n}\psi^-(\pi(\xi)) + C_0
\end{equation}
We construct analogously  $\cC^-(r)$ a \emph{Moran cover} of the space of one-sided sequences $\Sigma_M^-=\{1,\ldots,M\}^{-\bN}$ of parameter $r$ (relative to the function $\psi^-$) and denote by $N_1^-(r)$ the correspondingly defined positive integer. Concatenating all such cylinders, let
\[
	\cC(r)
	\eqdef\{[\eta_{-n}\ldots\eta_{-1}.\eta_0\ldots\eta_{m-1}]\colon [\eta_{-n}\ldots\eta_{-1}]\in\cC^-(r),[\eta_0\ldots\eta_{m-1}]\in\cC^+(r)\}.
\]
Given $C\in\cC(r)$ denote $\lvert C\rvert\eqdef r$.
Given $\rho>0$ put
\[	
	\cC_\rho
	\eqdef \bigcup_{r\in(0,\rho)}\cC(r).
\]
Observe that $N_1^+(r)$ and $N_1^-(r)$ diverge when $r\to0$.

The conjugation map $\pi\colon\Sigma_M\to\Gamma$ sends each cylinder $C=[\eta_{-n}\ldots \eta_{m-1}]\in\cC_\rho$ into a Markov rectangle 
\[
	R_{\eta_{-n}\ldots \eta_{m-1}}
	\eqdef \pi([\eta_{-n}\ldots\eta_{m-1}]),
\]	 
which has roughly (up to some constant which is universal on $\Gamma$ and which depends on the geometry of stable/unstable manifolds and on distortion estimates) lengths given by $r$ in the stable and the unstable directions, respectively. Denote
\[
	\cR_\rho
	\eqdef \bigcup_{r\in(0,\delta)}\cR(r),
	\quad\text{ where }\quad
	\cR(r)=\{\pi(C)\colon C\in\cC(r)\} 
\]
and for $R\in\cR(r)$ we wite $\lvert R\rvert\eqdef r$.

Consider the family $\cF=\cR_\rho$ and define the Hausdorff measure and dimension (with respect to $\cF$) (see Section~\ref{sec:dis}). By hyperbolicity of $\Gamma$, this family indeed satisfies the properties (HD1)--(HD3). Hence, given $\theta>\dim_\H( A)$ there exists $\rho\in(0,1)$ and a countable cover $\{R_i\}_i$ of $A$ by  rectangles $R_i\in\cR_\rho$  such that
\[
	N_1^-(\rho)\ge N_0,\quad 
	N_1^+(\rho)\ge N_0,\quad 
	\sum_i\lvert R_i\rvert^\theta \le1.
\]
To fix notation, note that every $R_i$ is in $\cC(r_i)$ for some $r_i\in(0,\rho)$ and is defined by means of some corresponding finite sequence $(\eta^i_{-n_i^-}\ldots\eta^i_{-1}.\eta^i_0\ldots\eta^i_{n_i^+-1})\in\cC(r_i)$. 
By construction of the Moran cover of parameter $r_i$ and by~\eqref{eq:limsupunif}, for every $R_i$ we have ($\eta^i$ to be taken some arbitrary infinite sequence in the cylinder $[\eta^i_{-n_i^-}\ldots\eta^i_{-1}.\eta^i_0\ldots\eta^i_{n_i^+-1}]$)
\[
	\log\lvert R_i\rvert
	> {S_{n_i^+}\psi^+(\pi(\eta^i))}
	> {-(\chi^++2\varepsilon)n_i^+}.
\]
Thus, we can estimate
\begin{equation}\label{eq:coverent}
	\sum_ie^{-(\chi^++2\varepsilon)\theta n_i^+}
	< 1.
\end{equation}	

Note that~\eqref{eq:boundsss},~\eqref{eq:bounds-s}, and~\eqref{eq:limsupunif} together imply
\[\begin{split}
	n_i^- (\chi^- -2\varepsilon )
	&< S_{-n_i^-}\psi^-(\pi(\eta^i))
	< \log r
	\le S_{n_i^+}\psi^+(\pi(\eta^i)) + C_0
	<-n_i^+(\chi^+-2\varepsilon) + C_0
\end{split}\]
which implies
\begin{equation}\label{eq:onemore}
	- n_i^+\frac{\chi^+}{\chi^-}
	\leq n_i^- - \frac{2\varepsilon(n_i^- + n_i ^+) + C_0}{\chi^-}.
\end{equation}
Further, with Young's formula~\eqref{eq:Young} 
$$
	\dim_\H  \mu 
	= h_\mu(f) \left(\frac{1}{\chi^+} - \frac{1}{\chi^-}\right)
$$
with~\eqref{eq:onemore} we obtain
\[\begin{split}
	\theta n_i^+\chi^+ 
	&= h_\mu(f)\frac{\theta}{\dim_\H\mu}\left(n_i^+-n_i^+\frac{\chi^+}{\chi^-}\right) \\
	&\le h_\mu(f)\frac{\theta}{\dim_\H\mu}\left(n_i^+
		+n_i^- - \frac{2\varepsilon(n_i^- + n_i ^+) + C_0}{\chi^-}\right)\\
	&= 	h_\mu(f)(n_i^+ + n_i^-)\frac{\theta}{\dim_\H\mu}
		\left(1  + \frac{2\varepsilon}{\lvert\chi^-\rvert} 
			+ \frac{C_0}{\lvert\chi^-\rvert(n_i^+ + n_i^-)}\right).
\end{split}\]
By our hypotheses~\eqref{eq:choosee} and~\eqref{eq:chooseee} on $\varepsilon$ and $\delta$ we can conclude that 
\begin{equation}\label{eq:epsdel}
	\frac{\theta}{\dim_\H\mu}\left(1  + \frac{2\varepsilon}{\lvert\chi^-\rvert}
				 +\frac{C_0}{\lvert\chi^-\rvert(n_i^+ + n_i^-)}\right)
	< 1-\frac{\delta}{4}.
\end{equation}
Hence with $h_\mu(f)\le h(\sigma|_{\Sigma_M})$ and with~\eqref{eq:epsdel} and~\eqref{eq:choosee} we can conclude 
\[\begin{split}
	\theta n_i^+\chi^+ +2\varepsilon n_i^+
	&< (n_i^++n_i^-)h(\sigma|_{\Sigma_M})
		\left[1-\frac\delta4 +\frac{2\varepsilon n_i^+}{(n_i^++n_i^-)h(\sigma|_{\Sigma_M})}\right]	\\
	&< (n_i^++n_i^-)h(\sigma|_{\Sigma_M})
		\left[1-\frac\delta4 +\frac{2\varepsilon }{h(\sigma|_{\Sigma_M})}\right]	\\
\text{with }\eqref{eq:chooseee}\quad
	&< (n_i^++n_i^-)h(\sigma|_{\Sigma_M})
		\left[1-\frac\delta8\right]			
\end{split}\]
Consider now the family of cylinders of length $\lvert U_i\rvert=n_i^-+n_i^+$ given by
\[
	\cU=\{U_i\colon U_i=\sigma^{-n_i^-}(C_i)\}
	\subset \bigcup_{n\ge N_0}\Sigma_{M,n}^+.
\]
With~\eqref{eq:coverent} and the above estimates, this family satisfies
\[
	\sum_ie^{-s\lvert U_i\rvert}\le 1
	\quad\text{ for some  }
	s <  h(\sigma|_{\Sigma_M}).
\]
Hence, by Proposition~\ref{pro:Dol} we have $h(\sigma|_{\Sigma_M},I^+(\cU))\ge H(N_0)$ for some function $H$ satisfying $\lim_{n\to\infty}H(n)= h(\sigma|_{\Sigma_M})$. Note that 
\[
	I^+(\cU)\subset \pi(I^+_{g|\Gamma}(A\cap\Gamma)).
\]
Hence, monotonicity of entropy (E5) implies $h(g|_\Gamma,I^+_{g|\Gamma}(A\cap\Gamma))\ge H(N_0)$. 
When letting $N_0\to\infty$ we obtain $h(g|_\Gamma,I^+_{g|\Gamma}(A\cap\Gamma))=h(g|_\Gamma)$.

One verifies that $I^+_{g|\Gamma}(A\cap\Gamma)=I^+_{f|\Gamma}(A\cap\Gamma)$ and by (E4) hence $h(f|_\Gamma,I^+_{f|\Gamma}(A\cap\Gamma))=h(f|_\Gamma)$. This proves the proposition.
\end{proof}

\section{Proofs}\label{sec:final}

\begin{proof}[Proof of Theorem~\ref{the:smallsmall}]
	Let $\mu$ be a hyperbolic ergodic measure and $W$ its support which by hypothesis is locally maximal. Let $A\subset W$ some set with $\dim_\H( A)<\dim_\H\mu$.  
	
	Given $\varepsilon$, let $W_\varepsilon\subset W$ be a $(\mu,\varepsilon)$-horseshoe as provided by Lemma~\ref{lemaqeaprox1}. Since $W$ is  locally maximal, for $\varepsilon$ small we can assume $W_\varepsilon\subset W$. We have $d^\u(W_\varepsilon)+d^\s(W_\varepsilon)
	= \dim_\H( W_\varepsilon)$ and  by Lemma~\ref{lemaqeaprox1} 
\begin{equation}\label{criaria}
	h(f|_{W_\varepsilon})
	\ge h_\mu(f)-\varepsilon
	\quad\text{ and }\quad
	d^\star(W_\varepsilon)
	\ge \frac{h_\mu(f)}{\lvert\chi^\star(\mu)\rvert}-\delta(\varepsilon)
\end{equation}
for $\star=s,u$ respectively, where $\delta(\varepsilon)\to0$ as $\varepsilon\to0$.

By Proposition~\ref{pro:diment} and by Proposition~\ref{pro:Manning} applied to the set $B=I^+_{f|W_\varepsilon}(A\cap W_\varepsilon)$, for every $x\in W_\varepsilon$ we have
\begin{equation}\label{criariau}
	 h(f|_{W_\varepsilon})
	 = h\big(f|_{W_\varepsilon},I^+_{f|W_\varepsilon}(A\cap W_\varepsilon)\big)
	 = h\big(f|_{W_\varepsilon},I^+_{f|W_\varepsilon}(A\cap W_\varepsilon)\cap\cW^\u_{\rm loc}(x,f)\big).
\end{equation}
Hence, by Proposition~\ref{pro:localdim} applied to $I^+_{f|W_\varepsilon}(A\cap W_\varepsilon)$  we have
\begin{equation}\label{eq:1}
	\dim_\H \big(I^+_{f|W_\varepsilon}(A\cap W_\varepsilon)\cap\cW^\u_{\rm loc}(x,f) \big)
	\ge \frac{\chi^\u(\mu)-\varepsilon}{\chi^\u(\mu)+\varepsilon}
		\dim_\H \big(W_\varepsilon\cap\cW^\u_{\rm loc}(x,f)\big).
\end{equation}
Recall that by~\eqref{eq:stunstdim} we have
\begin{equation}\label{eq:2}
	\dim_\H \big(W_\varepsilon\cap\cW^\u_{\rm loc}(x,f)\big)
	= d^\u(W_\varepsilon).
\end{equation}
By Lemma~\ref{lem:locally}, for every $y\in I^+_{f|W_\varepsilon}(A\cap W_\varepsilon)\cap\cW^\u_{\rm loc}(x,f)$ we have $I^+_{f|W_\varepsilon}(A\cap W_\varepsilon)\supset W_\varepsilon\cap\cW^\s_{\rm loc}(y,f)$, and the Hausdorff dimension of the latter is equal to $d^\s(W_\varepsilon)$ for every such $y$ (recall~\eqref{eq:localdimension}). 

As we consider a basic set of a surface diffeomorphism, the holonomy maps between stable (unstable) local manifolds are Lipschitz continuous. So locally and up to a Lipschitz continuous change of coordinates,  $W_\varepsilon$ is a direct product of slices taken with a local unstable and a local stable manifold, respectively  (see~\cite{PalVia:88} for details).  By~\cite[Theorem 1]{McCMan:83}, the Hausdorff dimension of such slices of $W_\varepsilon$ does not depend on the choice of manifolds (formulas~\eqref{eq:stunstdim}).
Now we apply Lemma~\ref{lem:Marstrand} to $B_1=I^+_{f|W_\varepsilon}(A\cap W_\varepsilon)\cap\cW^\u_{\rm loc}(x,f)$ taking $b_1$ hence provided by~\eqref{eq:1} together with~\eqref{eq:2}. To apply this lemma, we also take $B_2$ to be arcs in the local stable manifolds and take $b_2=d^\s(W_\varepsilon)$. By the fact that the exceptional set is $\s$-saturated and by the fact that any intersection of $W_\varepsilon$ with a local stable manifold by~\eqref{eq:stunstdim} has constant dimension $b_2\eqdef d^\s(W_\varepsilon)$, we obtain
\[
	\dim_\H( I^+_{f|W_\varepsilon}(A\cap W_\varepsilon))
	\ge d^\s(W_\varepsilon) 
	+  \frac{\chi^\u(\mu)-\varepsilon}{\chi^\u(\mu)+\varepsilon}d^\u(W_\varepsilon).
\]
Observe that $\dim_\H( I^+_{f|W}(A))\ge \dim_\H( I^+_{f|W_\varepsilon}(A\cap W_\varepsilon))$. As $\varepsilon$ was arbitrary, with~\eqref{criaria} and~\eqref{criariau} we conclude
\[
	h(f|_W,I^+_{f|W}(A))
	\ge h_\mu(f)
	\quad\text{ and }\quad
	\dim_\H( I^+_{f|W}(A))
	\ge \dim_\H\mu,
\].

Finally, to obtain the estimates for the (possibly smaller subset) $E^+_{f|W}(A)$, observe that by Remark \ref{remarkIE} we have
\[
	\dim_\H  \mu  
	\leq \dim_\H( I^+_{f|W}(A) )
	= \max_{n\ge 0}\Big\{ \dim_\H( E^+_{f|W}(A)), 
	\dim_\H \big(f^{-n}(\tilde{A})\big)\Big\},
\] 
where $\tilde A\subset A$ was defined in Remark \ref{remarkIE}. 
Since $f$ is bi-Lipschitz, by property (H3) we have $\dim_{\rm  H} f^{-n}(\tilde{A}) = \dim_\H ( \tilde{A})$ for every $n\ge0$. Since hence $\dim_\H\big(  f^{-n}\big(\tilde{A}\big)\big)= \dim_\H\big(  \tilde{A}\big)\leq \dim_\H(  A) < \dim_\H \mu$ 
and since we already proved $\dim_{\text{H}}(I^+_{f|W}(A))\geq \dim_\H \mu$, 
this implies 
\[
	\dim_\H( E^+_{f|W}(A)) \geq \dim_\H \mu.
\]	
This finishes the proof of the theorem.
\end{proof}

\begin{proof}[Proof of Theorem~\ref{the:basic}]
We can proceed exactly as in the proof of Theorem~\ref{the:smallsmall}, considering $(\mu,\varepsilon)$-horseshoes $W_\varepsilon\subset\Gamma$.  

The only difference is the application of the slicing argument. By Lemma~\ref{lem:locally}, for every $y\in I^+_{f|W_\varepsilon}(A\cap W_\varepsilon)\cap\cW^\u_{\rm loc}(x,f)$ we have $I^+_{f|\Gamma}(A)\supset\Gamma\cap\cW^\s_{\rm loc}(y,f)$, and the Hausdorff dimension of the latter is equal to $d^\s(\Gamma)$ for every such $y$ (recall~\eqref{eq:localdimension}). Then as before we can consider the local product structure of $\Gamma$ and by Lemma~\ref{lem:Marstrand} with $b_1=d^\s(W_\varepsilon)$ we can conclude that
\[
	\dim_\H( I^+_{f|\Gamma}(A))
	\ge d^\s(\Gamma)+\frac{\chi^\u(\mu)-\varepsilon}{\chi^\u(\mu)+\varepsilon}d^\u(W_\varepsilon).
\]
As $\varepsilon$ was arbitrary, with~\eqref{criaria} we conclude
\[
	\dim_\H( I^+_{f|\Gamma}(A))
	\ge d^\s(\Gamma) + \frac{h_\mu(f)}{\chi^\u(\mu)},
\]
which finishes the proof.
\end{proof}

\begin{proof}[Proof of  Theorem \ref{main}]
Consider the sequences $(\mu_n)_n$, $(\varepsilon_n)_n$ and $(\Gamma_n)_n$ provided by Lemma \ref{lemaqeaprox} such that, in particular $\lim_n\dim_\H\mu_n=\DD(f|_W)$. 
By  hypothesis we have 
\[
	\dim_\H (A)<\DD(f|_W).
\]
Hence, for $n$ sufficiently large we have (the first inequality is simple)
\[
	\dim_\H (A\cap \Gamma_n)
	\le \dim_\H (A)
	< \dim_\H\mu_n
	\le \DD(f|_{\Gamma_n})
	\le \DD(f|_W).
\]
From Theorem~\ref{the:smallsmall} we obtain $\dim_\H( I^+_{f|W}(A))\ge\dim_\H\mu_n$.
Now letting $n\to\infty$ implies
\[
	\dim_\H( I^+_{f|W}(A) )
	\ge \DD(f|_W)
\]
as claimed.

To estimate the dimension of $E^+_{f|W}(A)$,  by Remark \ref{remarkIE} 
we can conclude
\[
	\DD(f|_W) 
	\leq \dim_\H( I^+_{f|W}(A) )
	= \max_{n\ge 0}\Big\{ \dim_\H (E^+_{f|W}(A)), 
	\dim_\H \big(f^{-n}(\tilde{A})\big)\Big\}
\] 
Since $f$ is bi-Lipschitz, by (H3) we  have 
$\dim_\H ( f^{-n}(\tilde{A})) = \dim_\H ( \tilde{A})$. 
This implies that for every $n\ge0$ we have
$\dim_\H( f^{-n}\big(\tilde{A}\big))= \dim_\H ( \tilde{A})\leq \dim_\H(  A) < \DD (f|_W)$. Together with  $\dim_{\text{H}}I^+_{f|W(A)}\geq \DD (f|_W)$ this implies $\dim_\H( E^+_{f|W}(A)) \geq \DD(f|_W)$.
\end{proof}

\bibliographystyle{novostyle}

\end{document}